\documentclass[12pt,a4paper]{amsart}
\usepackage[all]{xy}
\newtheorem{thm}{Theorem}[section]
\newtheorem{lem}[thm]{Lemma}
\newtheorem{cor}[thm]{Corollary}
\newtheorem{prop}[thm]{Proposition}

\theoremstyle{definition}
\newtheorem{defn}[thm]{Definition}

\newtheorem{rem}[thm]{Remark}

\numberwithin{equation}{section}

\newcommand{\thmref}[1]{Theorem~\textup{\ref{#1}}}

\newcommand{\lemref}[1]{Lemma~\textup{\ref{#1}}}
\newcommand{\propref}[1]{Proposition~\textup{\ref{#1}}}

\newcommand{\remref}[1]{Remark~\textup{\ref{#1}}}

\newcommand{\CC}{\mathcal C}
\renewcommand{\AA}{\mathcal A}

\newcommand{\supp}{\operatorname{supp}}

\renewcommand{\a}{\alpha}
\renewcommand{\b}{\beta}
\renewcommand{\d}{\delta}

\newcommand{\p}{\phi}

\newcommand{\s}{\sigma}
\newcommand{\g}{\gamma}
\renewcommand{\o}{\omega}
\newcommand{\D}{\Delta}
\renewcommand{\t}{\theta}
\renewcommand{\P}{\Phi}

\newcommand{\pp}{\mathfrak p}
\newcommand{\nn}{\mathfrak n}
\newcommand{\mm}{\mathfrak m}

\DeclareMathOperator{\aut}{Aut}
\DeclareMathOperator{\ad}{Ad}
\DeclareMathOperator{\fix}{Fix}
\newcommand{\dn}{\downarrow}
\newcommand{\iso}{\overset{\cong}{\longrightarrow}}
\newcommand{\what}{\widehat}
\newcommand{\wilde}{\widetilde}

\DeclareMathOperator*{\spn}{span}
\DeclareMathOperator*{\clspn}{\overline{\spn}}
\DeclareMathOperator{\obj}{Obj}
\DeclareMathOperator{\mor}{Mor}

\newcommand{\midtext}[1]{\quad\text{#1}\quad}
\newcommand{\righttext}[1]{\quad\text{#1 }}
\renewcommand{\and}{\midtext{and}}

\newcommand{\id}{\text{\textup{id}}}
\newcommand{\rt}{\textup{rt}}
\newcommand{\lt}{\textup{lt}}
\newcommand{\DD}{\mathcal D}
\newcommand{\RCP}{\textup{RCP}}
\newcommand{\CP}{\textup{CP}}
\newcommand{\red}{\textup{Red}}
\newcommand{\nor}{\textup{Nor}}
\newcommand{\aug}{\textup{aug}}

\renewcommand{\>}{\rangle}

\begin{document}

\title[Proper actions and fixed-point algebras]
{Proper actions, fixed-point algebras\\ and naturality in nonabelian duality}



\author[Kaliszewski]{S. Kaliszewski}
\address{Department of Mathematics and Statistics
\\Arizona State University
\\Tempe, Arizona 85287, USA}
\email{kaliszewski@asu.edu}
\author[Quigg]{John Quigg}
\address{Department of Mathematics and Statistics
\\Arizona State University
\\Tempe, Arizona 85287, USA}
\email{quigg@asu.edu}
\author[Raeburn]{Iain~Raeburn}
\address{Iain Raeburn, School  of Mathematics and Applied Statistics, University of Wollongong, NSW 2522, Australia}
\email{raeburn@uow.edu.au}
\date{17 December 2007}

\begin{abstract}
Suppose a locally compact group $G$ acts freely and properly on a locally compact Hausdorff space $X$, and let $\g$ be the induced action on $C_0(X)$. We consider a category in which the objects are $C^*$-dynamical systems $(A,G,\a)$ for which there is an equivariant homomorphism of $(C_0(X),\g)$ into the multiplier algebra $M(A)$. Rieffel has shown that such systems are proper and saturated, and hence have a generalized fixed-point algebra $A^\a$ which is Morita equivalent to $A\times_{\a,r}G$. We show that the assignment $(A,\a)\mapsto A^\a$ is functorial, and that Rieffel's Morita equivalence is natural in a suitable sense. We then use our results to prove a categorical version of Landstad duality which characterizes crossed products by coactions, and to prove that Mansfield imprimitivity for crossed products by homogeneous spaces is natural.
\end{abstract}

\thanks{Iain Raeburn was supported by the Australian Research Council, through the ARC Centre for Complex Dynamic Systems and Control, and by the Centre de Recerca Matem{\` a}tica at the Universitat Aut{\` o}noma de Barcelona.}

\maketitle

\section*{Introduction}

An action $\a$ of a locally compact group $G$ on a non-unital $C^*$-algebra $A$ typically has very few fixed-points, but there has nevertheless been considerable interest in situations where one can find a reasonable analogue of the fixed-point algebra inside the multiplier algebra $M(A)$ \cite{proper, exel, meyer, integrable, HRW05}. To see what one might hope to achieve, consider the action $\g$ of $G$ on $C_0(X)$ induced by an action of $G$ on a locally compact space $X$. When $G$ acts properly on $X$, the orbit space $X/G$ is locally compact and Hausdorff, and the algebra $C_0(X/G)$ behaves much like a fixed-point algebra for the system. For example, if $G$ acts freely as well as properly on $X$, then a theorem of Green \cite{g2} says that $C_0(X/G)$ is Morita equivalent to the crossed product $C_0(X)\times_{\g} G$. The crucial feature of this situation is that there is a dense $\g$-invariant $*$-subalgebra $C_c(X)$ where one can average functions $f$ over the orbits to obtain a function $E(f):x\cdot G\mapsto \int_G f(x\cdot t)\,dt$ in $C_0(X/G)$.

There are many other situations in which variations of Green's construction yield useful Morita equivalences. In search of a systematic approach to such equivalences, Rieffel studied a family of \emph{proper actions} $(A,\a)$ for which there is a dense invariant $*$-subalgebra $A_0$ with properties like those of the subalgebra $C_c(X)$ of $(C_0(X),\g)$, and for which there is a \emph{generalized fixed-point algebra} $A^\a$ in $M(A)$ \cite{proper}. He also identified a class of \emph{saturated} proper actions for which $A^\a$ is Morita equivalent to the reduced crossed product $A\times_{\a,r}G$ (see \cite[Corollary~1.7]{proper}). Rieffel's theory was developed with some new examples in mind \cite{rieff89}, and it has since had other applications (see, for example, \cite{abadie}, \cite{landrae} and \cite{pr}). In particular, an Huef and Raeburn have used Rieffel's results to extend Mansfield's imprimitivity theorem for crossed products by coactions to crossed products by homogeneous spaces \cite{HR:mansfield}.

On the face of it, Rieffel's generalized fixed-point algebra $A^\a$ and the associated Morita equivalence depend on the choice of dense subalgebra $A_0$. However, Rieffel subsequently showed in \cite[Theorem~5.7]{integrable} that if $G$ acts properly on $X$ and there is an equivariant homomorphism $\phi:(C_0(X),\g)\to (A,\alpha)$, then $\alpha$ is proper in the sense of \cite{proper} with respect to $A_0=\spn C_c(X)AC_c(X)$. In the light of the work of the first two authors in \cite{clda}, Rieffel's hypothesis says that the system $(A,\a)$ belongs to a \emph{comma category} associated to the system $(C_0(X),\g)$, and his theorem says that, for systems in this comma category, there is a canonical choice of dense subalgebra $A_0$.

In this paper we show that when $G$ acts properly on $X$, Rieffel's generalized fixed-point algebra gives a functor $\fix$ on the comma category associated to $(C_0(X),\g)$, and that when the action of $G$ on $X$ is free, Rieffel's Morita equivalence gives a natural equivalence between $\fix$ and a reduced crossed-product functor. We then give several applications of these results to nonabelian duality. The first (taking $X=G$) is a categorical version of Landstad duality for coactions \cite{Q:landstad}, which identifies crossed products by coactions as the dynamical systems which belong to the comma category associated to $(C_0(G),\g)$; our result neatly complements those of \cite{clda} for actions,
and when
combined with them
gives new information about iterated Landstad duality. In our second application, we take $(X,G)=(G,H)$ where $H$ is a closed subgroup of $G$, and deduce that the Morita equivalence implementing the imprimitivity theorem for crossed products by homogeneous spaces in \cite{HR:mansfield} is natural; this result complements the result for coactions of quotient groups in \cite[Theorem~4.21]{enchilada}.

\smallskip

We begin with a short section on preliminaries, in which we introduce
(some of)
the categories of interest to us. In \S\ref{fix}, we investigate an averaging process $E$ developed in \cite{Q:landstad} for use in nonabelian duality. For systems $(A,\a,\phi)$ in the comma category, we can average elements $fag$ of $C_c(X)AC_c(X)$, and the elements $E(fag)$ span a $*$-algebra whose closure we denote by $\fix(A,\a,\phi)$. Our first main result (Corollary~2.8) says that $\fix$ extends to a functor on the comma category. In \S\ref{bimodules}, we prove that, when $(A,\a,\phi)$ is an element of the comma category, $(A,\a)$ is proper with respect to $A_0:=\spn C_c(X)AC_c(X)$ and has $\fix(A,\a,\phi)$ as its generalized fixed-point algebra. Theorem~\ref{XAnatequiv} is our main theorem on the naturality of Rieffel's Morita equivalence.

Our new version of categorical Landstad duality for crossed products by coactions is Theorem~\ref{Thmfromcldc}. Theorem~3.3 of \cite{Q:landstad} gives the required result on objects (see Theorem~\ref{Qld}), so our main task is to extend this to morphisms.
As in any project involving coactions, technical choices have to be made: here, the main references of interest to us \cite{proper, Q:landstad, HR:mansfield} use reduced crossed products and reduced coactions, so in an attempt to be user-friendly we have used reduced coactions wherever possible. Nevertheless, there are other options, and we have discussed them briefly in Corollary~\ref{normmax} and Remark~\ref{normmax2}, relegating the necessary background details to the Appendix.
In \S\ref{compare}, we combine our categorical Landstad duality for coactions with the duality for actions of \cite{clda}, and in \S\ref{sec-man} we apply our results to Mansfield imprimitivity.
The Appendix contains a new categorical version of the relationship between normal and reduced coactions which may be of independent interest,
as well as the facts about reduction and normalization of coactions needed in the main text.

\section{Preliminaries}
\label{prelim}

We begin with a convention: if $X$ and $Y$ are subspaces of a $C^*$-algebra $A$, then $XY$ denotes the subspace $\spn\{xy:x\in X,\ y\in Y\}$. If $A$ and $B$ are $C^*$-algebras, a homomorphism $A\to M(B)$ is \emph{nondegenerate} if $\phi(A)B$ is dense in $B$. Such  a homomorphism extends uniquely to a unital homomorphism $\phi:M(A)\to M(B)$ (see \cite[Corollary 2.51]{tfb}, for example), and hence nondegenerate homomorphisms can be composed. The objects in our basic category $\CC$ are $C^*$-algebras, and the morphisms from one object $A$ to another $B$ are nondegenerate homomorphisms $\phi:A\to M(B)$.

Throughout this paper, we consider a fixed locally compact group $G$, and various categories of actions and coactions of $G$ on $C^*$-algebras. The objects in the category $\AA(G)$ consist of dynamical systems $(A,\a)$ in which $\a$ is a continuous action of $G$ on $A$, and the morphisms from $(A,\a)$ to $(B,\b)$ are morphisms $\phi:A\to B$ in $\CC$ such that $\phi\circ\alpha_s=\beta_s\circ\phi$ for every $s\in G$. We will introduce the various categories of coactions as they are required.

By a ``natural isomorphism'' between two functors we mean exactly the same as we meant by a ``natural equivalence'' in \cite{enchilada}. We now believe that ``natural isomorphism'' is the more generally accepted term in category theory.

Let $a$ be an object in a category $C$.
The \emph{comma category} $a\downarrow C$ has objects $(x,f)$, where $f:a\to x$ is a morphism in $C$, and a morphism
$h:(x,f)\to (y,g)$ in $a\dn C$ is a morphism $h:x\to y$ in $C$ such that $h\circ
f=g$.

\begin{rem}\label{aug-rem}
We emphasize
that our underlying category $\CC$ is not the same as the category denoted by $\CC$ in \cite{enchilada}, which we will denote here by $\CC_{\aug}$, and refer to as the \emph{augmented category} of $C^*$-algebras. The objects in $\CC_{\aug}$ are again $C^*$-algebras, but in $\CC_{\aug}$ the morphisms from $A$ to $B$ are isomorphism classes of
$A$-$B$ correspondences (called right-Hilbert $A$-$B$ bimodules in \cite{enchilada}).
A morphism $\phi\colon A\to B$ in $\CC$ gives rise to an $A$-$B$ correspondence, and
hence to a morphism $[\phi]\colon A\to B$ in $\CC_{\aug}$.
Moreover, the assignment $\phi\mapsto [\phi]$ gives a functor from $\CC$ to $\CC_{\aug}$;
this point will be important in~\S\ref{bimodules} (see Remark~\ref{intoBEcat}).
While this functor is neither injective nor surjective on morphisms,
we do have $[\phi]=[\psi]$ in $\CC_{\aug}$  if and only if $\psi=\ad u\circ\psi$ for some $u\in U\!M(B)$ \cite[Proposition~2.3]{taco}.
\end{rem}

\section{Fixed points}
\label{fix}

In this section, we consider a proper right action of $G$ on a locally compact space $X$,  the associated action $\gamma:G \to\aut C_0(X)$ given by $\gamma_t(f)(x)=f(x\cdot t)$, and an element $(A,\alpha,\phi)=((A,\a),\p)$ of the comma category $(C_0(X),\g)\dn \AA(G)$. The morphism $\phi$ is implemented by a nondegenerate homomorphism $\phi:C_0(X)\to M(A)$, and $A$ thus becomes a bimodule over $C_0(X)$; to simplify notation, we often write $fag$ for $\phi(f)a\phi(g)$. Our goal in this section is to prove that applying the averaging process $E$ of \cite[Section~3]{Q:landstad} to the dense subalgebra
\begin{equation}\label{defA0}
A_0:=C_c(X)AC_c(X):=\spn\{fag:f,g\in C_c(X),\ a\in A\}
\end{equation}
of $A$ gives a $C^*$-subalgebra $\fix(A,\alpha,\phi):=\overline{E(A_0)}$ of $M(A)$, and that this is the object map of a functor $\fix$ from $(C_0(X),\g)\dn \AA(G)$ to $\CC$.

We begin by recalling the properties of the averaging process of \cite[Section~3]{Q:landstad}, which is based on \cite[Section~2]{connes} and \cite[note added in
proof]{connes2}. Our first comments apply to arbitrary systems $(A,\a)$ in $\AA(G)$.

As in \cite[Definition~3.4]{Q:landstad}, we let $\pp$ be the set of multipliers $a\in M(A)^+$ such that there exists an element $E(a)$ in $M(A)^+$ satisfying
\[
\omega(E(a))=\int \omega(\a_s(a))\,ds\]
for every positive functional $\omega$ on $A$. We know from \cite[note added in
proof]{connes2} and
\cite[Lemma~3.5]{Q:landstad} that $\nn:=\{a\in M(A):a^*a\in\pp\}$ is a left ideal in $M(A)$, and that $\mm:=\nn^*\nn$ is a $*$-subalgebra of $M(A)$ with $\mm^+=\pp$ and $\mm=\spn \pp$. Corollary~3.6 of \cite{Q:landstad} says that $a\mapsto E(a)$ extends to a positive linear
map $E=E^\a:\mm\to  M(A)$ such that
\begin{equation}\label{defE}
\omega(E(a))=\int \omega(\a_s(a))\,ds\righttext{for all $\omega\in A^*$.}
\end{equation}
Since $\nn$ is a left ideal in $M(A)$, we have
\[
\mm M(A)\mm=\nn^*(\nn M(A)\nn^*)\nn\subset \nn^* M(A)\nn\subset\mm,
\]
and for fixed $b,c\in \mm$, the map $a\mapsto E(bac)$ is norm-continuous on $M(A)$ \cite[Corollary~3.6(3)]{Q:landstad}.

\begin{lem}
\label{fixed}
If $d\in
M(A)^\a$
then for all $a\in \mm$ we have $ad\in \mm$ and $E(ad)=E(a)d$.
\end{lem}

\begin{proof}
We first claim that $\nn d\subset \nn$. Let $a\in \nn$, so that $a^*a\in \pp$.
To show $ad\in \nn$ we need $d^*a^*ad\in \pp$, so by definition of $\pp$ we need
to show that there exists $b\in M(A)^+$ such that
\[
\omega(b)=\int \omega(\a_s(d^*a^*ad))\,ds\righttext{for all $\omega\in A^*{}^+$.}
\]
We show that $b:=d^*E(a^*a)d$ works. For $\o\in A^*{}^+$, we define $d\cdot\omega\cdot d^*(b)=\o(d^*bd)$. Then $d\cdot\omega\cdot d^*$ is in $A^*{}^+$, and hence
\begin{align*}
\omega(d^*E(a^*a)d)
&=d\cdot\omega\cdot d^*(E(a^*a))\\
&=\int d\cdot \omega\cdot d^*(\a_s(a^*a))\,ds\\
&=\int \omega(\a_s(d^*a^*ad))\,ds\righttext{(because $\a_s(d)=d$).}
\end{align*}

Now the formula $\mm=\nn^*\nn$ implies that $\mm d\subset \mm$, and for $a\in
\mm$ and $\omega\in A^*$ we have
\begin{align*}\omega(E(ad))
&=\int \omega(\a_s(ad))\,ds
=\int d\cdot \omega(\a_s(a))\,ds\\&=d\cdot \omega(E(a))
=\omega(E(a)d).
\qedhere
\end{align*}
\end{proof}

We can apply all this to the system $(C_0(X),\g)$. Then for each positive function $f$ in the subalgebra $C_c(X)$ and each $x\in X$, the function $t\mapsto f(x\cdot t)$ has compact support, and it follows from Fubini's theorem that $f\in \pp$ and $E(f)$ is pointwise multiplication by the function $x\mapsto \int_Gf(x\cdot t)\,dt$. It follows that $C_c(X)\subset \mm$, with $E(f)$ given by the same formula.

When $(A,\a,\phi)$ is in the comma category $(C_0(X),\g)\dn\AA(G)$, Proposition~1.4 of \cite{qr:induced} implies that $\phi(C_c(X))\subset \mm$.  Thus
\[
C_c(X)M(A) C_c(X)\subset \mm M(A)\mm\subset \mm,
\]
and the maps $a\mapsto E(fag)$ are norm-continuous for every fixed pair $f,g\in C_c(X)$.

We think of the expectation $E$ as being given by a strictly convergent integral $\int\alpha_s(a)\,ds$, and then the next lemma says that, when the usual norm-convergent $A$-valued integral also makes sense, the two coincide.
\begin{lem}\label{E=int}
For $a\in A$ and $f,g,h\in C_c(X)$, the function $s\mapsto f\a_s(gah)$ belongs to $C_c(G,A)$, and its $A$-valued integral satisfies
\[
\int_G f\a_s(gah)\,ds=fE(gah).
\]
\end{lem}

\begin{proof}
We have $f\a_s(gah)=f\g_s(g)\a_s(a)\g_s(h)$, and $f\g_s(g)=0$ unless $\supp f$ intersects $(\supp g)\cdot s^{-1}$; since $G$ acts properly on $X$,
\[
\{s\in G:(\supp f)\cap(\supp g)\cdot s^{-1} \not=\emptyset\}
\]
is contained in a compact set. Thus $s\mapsto f\a_s(gah)$ has compact support, and the integral $\int_G f\a_s(gah)\,ds$ gives a well-defined element of $A$ (as in \cite[Lemma~C.3]{tfb}, for example). Now let $\o\in A^*$. Lemma~C.3 of \cite{tfb} also implies that the bounded linear functional $\o$ pulls through the integral:
\[
\o\Big(\int_G f\a_s(gah)\,ds\Big)=\int_G \o(f\a_s(gah))\,ds.
\]
We can define $\o_f\in A^*$ by $\o_f(b)=\o(fb)$, and then \eqref{defE}\ gives
\begin{align*}
\int \o(f\a_s(gah))\,ds
&=\int \o_f(\a_s(gah))\,ds\\
&=\o_f(E(gah))\\
&=\o(fE(gah)).
\end{align*}
Since $\o\in A^*$ was arbitrary, the result follows.
\end{proof}

Slightly different versions of the next lemma were implicitly used in the proofs of \cite[Theorem~5.7]{integrable} and \cite[Theorem~4.4]{HRW05}.

\begin{lem}\label{key lemma}
For $a\in A$ and $f,g,h\in C_c(X)$, there exists $k\in C_c(X)$ such that
\[
fE(gah)=fE(gah)k.
\]
In particular, the subalgebra $A_0$ defined in \eqref{defA0} satisfies
\[
C_c(X) E(A_0)\subset A_0\and A_0 E(A_0)\subset A_0.
\]
\end{lem}

\begin{proof}
Lemma~\ref{E=int} implies that there is a compact set $K\subset G$ such that $f\a_s(gah)$ vanishes for $s$ outside a compact set $K\subset G$, so
\[
fE(gah)=\int_G f\a_s(gah)\,ds=\int_K f\a_s(ga)\g_s(h)\,ds,
\]
and then
the supports of the functions $\{\g_s(h):s\in K\}$ are all contained in a compact set $L\subset X$. We choose $k\in C_c(X)$ which is identically $1$ on $L$. Then $\g_s(h)=\g_s(h)k$ for all $s\in K$, and
\begin{align*}
fE(gah)&=\int_K f\g_s(g)\a_s(a)\g_s(h)k\,ds\\&=\Big(\int_K f\g_s(g)\a_s(a)\g_s(h)\,ds\Big)k\\&=fE(gah)k.
\end{align*}
The two inclusions follow easily.
\end{proof}

\begin{prop}\label{subalg}
$E(A_0)$ is a $*$-subalgebra of $M(A)$ which is contained in $M(A)^\a$.
\end{prop}

\begin{proof}
Since $E$ is linear and $*$-preserving, $E(A_0)$ is a $*$-closed linear subspace of $M(A)$. It is easy to check from the defining property \eqref{defE} that each $E(a)$ is invariant under $\a$. To see that
$E(A_0)$ is closed under multiplication, let $a,b\in A$ and $f,g,h,k\in C_c(X)$. Lemma~\ref{key lemma} gives $\ell\in C_c(X)$ such that $gE(hbk)=gE(hbk)\ell$, and then it follows from Lemma~\ref{fixed} that
\begin{align*}
E(fag)E(hbk)
=E(fagE(hbk))
=E(fagE(hbk)\ell).
\end{align*}
To finish off, observe that $E(fagE(hbk)\ell)$ is in $E(A_0)$ because $agE(hbk)$ is in $A\cdot C_c(X) \cdot M(A)\subset A$.
\end{proof}

\begin{defn}
For an object $(A,\a,\p)$ of $(C_0(X),\g)\dn \AA(G)$
we define $\fix A=\fix(A,\a,\phi)$ to be the norm closure of $E(A_0)$ in $M(A)$.
\end{defn}

By Proposition~\ref{subalg}, $\fix(A,\a,\phi)$ is a $C^*$-subalgebra of $M(A)^\a$.
Our next goal is to prove that $\fix$ extends to a functor from the comma category $(C_0(X),\gamma)\dn \AA(G)$ into $\CC$, and for this we need to know that morphisms respect the construction.

\begin{prop}
\label{fixonmorph}
Let $\s:(A,\a,\p)\to (B,\b,\psi)$ be a morphism in the comma category $(C_0(X),\g)\dn \AA(G)$.
Then $\s(A_0)E(B_0)\subset B_0$, and $\s$ restricts to a nondegenerate homomorphism $\s|$ of $\fix(A,\a,\phi)$ into $M(\fix(B,\b,\psi))$.
\end{prop}

The proof of nondegeneracy is surprisingly subtle, and depends on some properties of the averaging process for the system $(C_0(X),\g)$.
For $f\in C_c(X)$, $E^\g(f)$ is multiplication by the function $x\mapsto \int f(x\cdot s)\,ds$. The expectation $E^\g$ is related to $E^\alpha$ by \cite[Proposition~1.4]{qr:induced}, which implies that $\p$ maps $C_c(X)$ into the subalgebra $\mm\subset M(A)$ and satisfies $E^\a\circ\phi=\phi\circ E^\g$. We need the following standard lemma.

\begin{lem}
\label{Eg=1}
For each $h\in C_c(X)$ there exists $g\in C_c(X)$ such that $E(g)h=h$.
\end{lem}

\begin{proof}
Let $K=\supp h$, and choose $g_1\in C_c(X,[0,\infty))$ such that $g_1>0$ on $K$.
Then $E(g_1)>0$ on $K\cdot G$, so there exists $g_2\in C_c(X/G)$ such that $g_2=1/E(g_1)$ on $K\cdot G$.
Then $g:=g_1g_2$ satisfies
\[
E(g)(x\cdot G) =\int g_1(x\cdot s)g_2(x\cdot G)\,ds=E(g_1)(x\cdot G)g_2(x\cdot G),
\]
which is $1$ whenever $x\in K$.
\end{proof}

\begin{proof}[Proof of Proposition~\ref{fixonmorph}]
Let $f,g\in C_c(X)$, $a\in A$, and $b\in B_0$. By Lemma~\ref{key lemma}, there exists $h\in C_c(X)$ such that $gE(b)=gE(b)h$. Then
\begin{align*}
\s(fag)E(b)=f\s(a)gE(b)=f\s(a)gE(b)h.
\end{align*}
Since $C_c(X) E(B_0)\subset B_0$ by Lemma~\ref{key lemma}, we have $gE(b)\in B_0$, $\s(a)gE(b)\in B$, and $f\s(a)gE(b)h\in B_0$, justifying the first assertion.

Next we claim that $\s(E(A_0))E(B_0)\subset E(B_0)$.
Indeed, for $a\in A_0$ and $b\in B_0$,
\cite[Proposition~1.4]{qr:induced} and
Lemma~\ref{fixed} give
\begin{align*}
\s(E(a))E(b)=E(\s(a))E(b)=E(\s(a)E(b)),
\end{align*}
which belongs to $E(B_0)$ by the first assertion. Taking adjoints gives $E(B_0)\s(E(A_0))\subset E(B_0)$. Since $E(B_0)$ is dense in $\fix(B,\b,\psi)$, it follows that $E(a)$ multiplies $\overline{E(B_0)}=\fix(B,\b,\psi)$, and $\sigma$ maps $E(A_0)$ into $M(\fix(B,\b,\psi))$; continuity now implies that $\s$ maps $\fix(A,\a,\phi)$ into $M(\fix(B,\b,\psi))$.

To establish nondegeneracy, let $h,k\in C_c(X)$ and $b\in B$.
By Lemma~\ref{Eg=1}, there exists $g\in C_c(X)$ such that $E(g)h=h$; choose $f\in C_c(X)$ such that $fg=g$. Lemma~\ref{key lemma} implies that there exists $\ell\in C_c(X)$ such that $gE(hbk)=gE(hbk)\ell$. Then two applications of Lemma~\ref{fixed} show that
\begin{align*}
E(hbk)
&=E(E(fg)hbk)
=E(fg)E(hbk)
\\&=E(fgE(hbk))
=E(fgE(hbk)\ell).
\end{align*}
Lemma~\ref{key lemma} implies that $gE(hbk)$ belongs to $B$, so the nondegeneracy of $\s$ implies that there exists $a\in A$ such that $gE(hbk)\approx \s(a)gE(hbk)$ in norm in $B$.
Thus, by the continuity of $c\mapsto E(fc\ell)$, we have a norm approximation
\begin{align*}
E(hbk)&\approx E(f\s(a)gE(hbk)\ell)\\
&=E(f\s(a)gE(hbk))\\
&=\s(E(fag))E(hbk),
\end{align*}
where at the last stage we used Lemma~\ref{fixed} again.
Since $E(B_0)$ is dense in $\fix(B,\b,\psi)$, this approximation implies that $\s|$ is nondegenerate.
\end{proof}

\begin{cor}
\label{fix functor}
Suppose that $G$ acts properly on $X$, and $\gamma$ is the corresponding action of $G$ on $C_0(X)$.  Then the assignments
\[
(A,\a,\p)\mapsto \fix(A,\a,\p)\and \s\mapsto \s|
\]
give a functor from $(C_0(X),\g)\dn \AA(G)$ to $\CC$.
\end{cor}

\begin{proof}
It is obvious that $\s\mapsto \s|$ respects identity morphisms and compositions.
\end{proof}

\section{Naturality of Rieffel's Morita equivalence}
\label{bimodules}

We again fix a proper right action of $G$ on a locally compact space $X$,
let $\g$ be the associated action on $C_0(X)$,
and consider a system $(A,\a,\p)$ in the comma category $(C_0(X),\g)\dn \AA(G)$. The next proposition is a new version of Theorem~5.7 of \cite{integrable} in which we use our averaging process $E$ in place of Rieffel's operator-valued weight $\psi_\alpha$, and thereby identify Rieffel's generalized fixed-point algebra as our $\fix(A,\a,\phi)$.

\begin{prop}\label{is-proper}
For any object $(A,\alpha,\phi)$ in $(C_0(X),\gamma)\dn\AA(G)$,
$\alpha$ is proper with respect to $A_0 = C_c(X)AC_c(X)$
in the sense of \cite[Definition~1.2]{proper}\textup;
the generalized fixed-point algebra $A^\alpha$ is $\fix(A,\alpha,\phi)$,
and the right inner product on $A_0$ is given by
\[\<a,b\>_R=E(a^*b).\]
\end{prop}

\begin{proof}
The nondegeneracy of $\phi$ implies that the $*$-subalgebra $A_0$ is dense in $A$, and it is obviously $\alpha$-invariant. For $a,b\in A_0$
the map $s\mapsto a\a_s(b^*)$ is in $C_c(G,A)$, which implies that both $s\mapsto a\a_s(b^*)$ and $s\mapsto a\a_s(b^*)\D(s)^{-1/2}$
are in $L^1(G,A)$, as required in \cite[Definition~1.2(1)]{proper}.

Now fix $a,b,c\in A_0$. We observed in Proposition~\ref{subalg} that $E(a^*b)$ is in $M(A)^\a$, and the inclusion $A_0E(A_0)\subset A_0$ in Lemma~\ref{key lemma} implies that $E(a^*b)$ multiplies $A_0$. Lemma~\ref{E=int} implies that the function
$s\mapsto c\a_s(a^*b)$
is in $C_c(G,A)$ and
\[cE(a^*b)=\int c\a_s(a^*b)\,ds.\]
So $\langle a,b\rangle_R:=E(a^*b)$ has the properties required in \cite[Definition~1.2(2)]{proper}. Thus $\a$ is proper.

To identify $A^\a$, we note first that
\[
A^\alpha = \clspn\{ E(a^*b) \mid a,b\in A_0 \}\subset \overline{E(A_0)}=\fix(A,\alpha,\phi).
\]
For the reverse inclusion, let $fag\in A_0$. By the nondegeneracy of $\phi$, we can approximate
$a\approx a_1^* hk a_2$,
and then by continuity of $a\mapsto E(fag)$ we have $E(fag) \approx E(fa_1^*hka_2g)$, and $A^\alpha \supset \fix(A,\alpha,\phi)$.
\end{proof}

In \cite{integrable}, Rieffel asserted that if the proper action of $G$ on $X$ is free, then  $\a$ is saturated with respect to $A_0$ in the sense of \cite[Definition~1.6]{proper}, and a proof of this assertion was provided in \cite[Lemma~4.1]{HRW05}. Thus Corollary~1.7 of
\cite{proper} implies that the completion $Z(A)$ of $A_0$ in the norm defined by $\langle \cdot,\cdot\rangle_R$ is an $(A\times_{\a,r} G)$--$\fix(A,\a,\phi)$ imprimitivity bimodule. The left module action of $C_c(G,A)\subset A\times_{\a,r} G$ on $A_0\subset Z(A)$ is given by
\[f\cdot a=\int_G f(s)\a_s(a)\D(s)^{1/2}\,ds,\]
and the left inner product ${}_L\<a,b\>$ for $a,b\in A_0$ is the element of $C_c(G,A)$ defined by
\[{}_L\<a,b\>(s)=a\a_s(b^*)\D(s)^{-1/2}.\]

\begin{thm}\label{XAnatequiv}
Suppose that $G$ acts freely and properly on a locally compact space $X$, and $\s:(A,\a,\p)\to (B,\b,\psi)$ is a morphism in $(C_0(X),\g)\dn \AA(G)$. Then the diagram
\begin{equation}\label{natofRieffel}
\xymatrix@C+30pt{
A\times_{\a,r} G \ar[r]^-{Z(A)} \ar[d]_{\s\times_r G}
&\fix(A,\a,\phi) \ar[d]^{\s|}
\\
B\times_{\b,r} G \ar[r]_-{Z(B)}
&\fix(B,\b,\psi)
}
\end{equation}
of $C^*$-correspondences
commutes in the sense that
there is an isomorphism
\[\P:Z(A)\otimes_{\fix(A,\a,\phi)} \fix(B,\b,\psi)\iso Z(B)\]
of $(A\times_{\a,r} G) - \fix(B,\b,\psi)$ correspondences
such that
\[\P(a\otimes E(b))=\s(a)E(b)
\righttext{for $a\in A_0$ and $b\in B_0$.}\]
\end{thm}
\begin{rem}\label{intoBEcat}
By composing with the functor $\phi\mapsto [\phi]$ from $\CC$ to $\CC_{\aug}$
described in Remark~\ref{aug-rem}, we may view both $\fix$ and
the reduced-crossed-product functor
$\RCP:(A,\a,\p)\mapsto A\times_{\a,r} G$ as
taking values in
the augmented category $\CC_{\aug}$ of~\cite{enchilada}.
Theorem~\ref{XAnatequiv} then says that diagram~\eqref{natofRieffel}
commutes in $\CC_{\aug}$, so
the assignment $(A,\alpha,\phi)\mapsto Z(A)$ implements a natural isomorphism between $\RCP$ and $\fix$.
\end{rem}

For the proof of Theorem~\ref{XAnatequiv}, we need a lemma.

\begin{lem}
\label{dense}
$C_c(X) E(A_0)$ is dense in $Z(A)$.
\end{lem}

\begin{proof}
The homomorphism $\phi:C_0(X)\to M(A)$, which is a morphism in the category $\CC$, induces a morphism
\[\p\times_r G:C_0(X)\times_{\g,r} G\to A\times_{\a,r} G
\]
in $\CC$. Since $A\times_{\a,r} G$ acts nondegenerately on the bimodule $Z(A)$, so does $C_0(X)\times_{\g,r} G$.  The action $\g$ of $G$ on $C_0(X)$ is saturated, which ensures that the inner products
\[
{}_L\<f,g\>(s)= f\g_s(\bar g)\D(s)^{-1/2}
\]
for $f,g\in C_c(X)\subset Z(C_0(X))$ span a dense subspace of $C_0(X)\times_{\g,r} G$. Thus the elements
\[
{}_L\<f,g\>a=\int_G f\g_s(\bar g)\a_s(a)\,ds
=\int_G f\a_s(\bar ga)\,ds
\]
for $f,g\in C_c(X)\subset Z(C_0(X))$ and $a\in A_0$ span a dense subspace of $Z(A)$. Since every $a\in A_0$ can be factored as $a=ga$, and
\[
fE(a)=fE(ga)=\int_G f\a_s(ga)\,ds={}_L\<f,\bar g\>a,
\]
we deduce that the elements $fE(a)$ span a dense subspace.
\end{proof}

\begin{cor}
\label{dense corollary}
If $\s:(A,\a,\p)\to (B,\b,\psi)$ is a morphism in the comma category $(C_0(X),\g)\dn \AA(G)$, then $\s(A_0)E(B_0)$ is dense in $Z(B)$.
\end{cor}

\begin{proof}
For $f,g\in C_c(X)$, $a\in A_0$, and $b\in B_0$ we have
\[\s(fag)E(b)=f\s(a)gE(b).\]
This suffices, because $C_c(X) E(B_0)$ is dense in $Z(B)$ by Lemma~\ref{dense}, and both $\s(A_0)$ and $C_c(X)$ act nondegenerately on $Z(B)$.
\end{proof}

\begin{proof}[Proof of Theorem~\ref{XAnatequiv}]
It is easily checked that $\Phi$ respects the right module action of $\fix (B,\b,\psi)$. To see that it preserves the inner products, we let $a,c\in A_0$, $b,d\in B_0$, and compute:
\begin{align*}
\langle \Phi(a\otimes E(b)),\Phi(c\otimes E(d))\rangle_R
&=E\big(E(b^*)\s(a^*c)E(d)\big)\\
&=E(b^*)E(\s(a^*c))E(d)\righttext{(by Lemma~\ref{fixed})}\\
&=E(b^*)\s(E(a^*c))E(d)\\
&=E(b^*)\s(\langle a,c\rangle_R) E(d)\\
&=\langle E(b),\langle a,c\rangle_R\cdot E(d)\rangle_R\\
&=\langle a\otimes E(b), c\otimes E(d)\rangle_R.
\end{align*}
To check that $\Phi$ preserves the left action, we note that the action of $f\in C_c(G,A)\subset A\times_{\a,r}G$ is given by a norm-convergent integral, which allows us to compute as follows:
\begin{align*}
\Phi(f\cdot (a\otimes E(b)))&=\s(f\cdot a)E(b)\\
&=\s\Big(\int f(s)\a_s(a)\D(s)^{1/2}\,ds\Big)E(b)\\
&=\int \s(f(s)\a_s(a))E(b)\D(s)^{1/2}\,ds\\
&=\int (\s\times_r G)(f)(s)\b_s(\s(a)E(b))\D(s)^{1/2}\,ds\\
&\hspace*{2.2in}\text{(since $E(b)\in M(B)^\beta$)}\\
&=(\s\times_r G)(f)\cdot \Phi(a\otimes E(b)).
\end{align*}
So $\Phi$ is a norm-preserving bimodule homomorphism. Corollary~\ref{dense corollary} says that $\Phi$ has dense range, and hence $\Phi$ is an isomorphism, as claimed.
\end{proof}

\begin{rem}
When $G$ does not act freely, the action $\g$ on $C_0(X)$ is not saturated, and we do not expect actions in the comma category to be saturated either. Rieffel's theory will then give a Morita equivalence between the generalized fixed-point algebra $\fix(A,\alpha,\phi)$ and an ideal $I(A)$ in $A\times_{\alpha,r}G$. It is tempting to conjecture that this equivalence is also natural, but it is not even obvious to us that the assignment $(A,\alpha,\phi)\mapsto I(A)$ is a functor in the necessary sense.
\end{rem}


\section{Landstad duality for coactions}\label{sec-ld}

Applying Corollary~\ref{fix functor} to the action $\rt:G\to \aut C_0(G)$ induced by right translation on $X=G$ gives a functor $\fix$ from $(C_0(G),\rt)\dn\AA(G)$ to $\CC$. As a map on objects, the functor $\fix$ underlies the Landstad duality for coactions in \cite[Theorem~3.3]{Q:landstad}, and we will augment the functor $\fix$ to give a categorical version of Landstad duality for coactions, parallel to the categorical Landstad duality for actions described in \cite{clda}.

The Landstad duality of \cite{Q:landstad} identifies the $C^*$-algebras which are isomorphic to crossed products by a coaction. The coactions in \cite{Q:landstad} are the \emph{reduced} coactions studied in \cite{lprs}, which are homomorphisms $\delta$ from a $C^*$-algebra $B$ to $M(B\otimes C_r^*(G))$ (see \cite[Definition~2.1]{lprs} for the full details). Theorem~3.7 of \cite{lprs} implies that the crossed product $B\times_\delta G$ is universal for a family of covariant homomorphisms $(\pi,\mu)$ consisting of nondegenerate homomorphisms $\pi:B\to M(C)$, $\mu:C_0(G)\to M(C)$ (where $C$ is any $C^*$-algebra) such that
\begin{equation}\label{defcov}
\pi\otimes\id(\delta(b))=\ad\mu\otimes\id(W_G)(\pi(b)\otimes 1)\righttext{for $b\in B$,}
\end{equation}
where we view $W_G:s\mapsto \lambda_s$ as a multiplier of $C_0(G,C_r^*(G))=C_0(G)\otimes C_r^*(G)$;
the $C^*$-algebra $B\times_\delta G$ is generated by a canonical covariant homomorphism $(j_B,j_G)$ in $M(B\times_\delta G)$. There is a natural dual action $\what\d:G\to \aut B\times_\d G$ which is characterized by
\[
\what\d_t(j_B(b)j_G(f))=j_B(b)j_G(\rt_t(f)).
\]

We can now restate Theorem~3.3 of \cite{Q:landstad} as follows:

\begin{thm}[\cite{Q:landstad}]\label{Qld}
Let $(A,\a,\phi)$ be an object in $(C_0(G),\rt)\dn\AA(G)$. Then
\[
\d(c):=\ad \p\otimes\id(W_G)(c\otimes 1)
\]
defines a reduced coaction $\d=\d^A$ of $G$ on $\fix A:=\fix(A,\a,\phi)$,
and there is an isomorphism  $\t=\t^A$ of $(\fix A)\times_\d G$ onto $A$ such that $\t\circ \a=\what\d$ and
\begin{equation}\label{charquiggiso}
\t\bigl(j_{\fix A}(c)j_G(f)\bigr)=c\p(f)
\righttext{for $c\in \fix A,\ f\in C_0(G)$.}
\end{equation}
\end{thm}

The reduced coactions $(B,\d)$ of $G$ are the objects in a category $\CC^r(G)$, whose morphisms
$\pi:(B,\d)\to (C,\epsilon)$
are the nondegenerate homomorphisms $\pi:B\to M(C)$ such that $(\pi\otimes\id)\circ\d=\epsilon\circ \pi$.
If $\s:(A,\a,\phi)\to (B,\b,\psi)$ is a morphism in the comma category $(C_0(G),\rt)\dn \AA(G)$, then Proposition~\ref{fixonmorph} implies that $\s$ restricts to a morphism $\s|:\fix (A,\a,\phi)\to \fix (B,\b,\psi)$ in $\CC$. The restriction $\s|$ satisfies
\[
\s|\otimes \id(\phi\otimes\id(W_G))=(\s\circ\phi)\otimes\id(W_G)=\psi\otimes\id(W_G).
\]
Thus $(\s|\otimes\id)\circ\d^A=\delta^B\circ \s|$, so $\s|$
is a morphism
in the category $\CC^r(G)$.
In this way, we
extend $\fix$ to a functor $\fix^r$ from $(C_0(G),\rt)\dn \AA(G)$ to $\CC^r(G)$.

\begin{thm}\label{Thmfromcldc}
The functor $\fix^r:(C_0(G),\rt)\dn \AA(G)\to \CC^r(G)$
is a category equivalence, with quasi-inverse given by
the crossed-product functor $\CP^r$
which assigns
\[
(B,\d)\mapsto (B\times_\d G,\what\d,j_G)
\and
\pi\mapsto \pi\times G.
\]
\end{thm}

\begin{proof}
By abstract nonsense, as in \cite[\S1]{clda}, it suffices to check that
$\CP^r\circ \fix^r\cong \id$ and
$\CP^r$ is full and faithful.
For the first statement, we must show that the isomorphism $\t^A$ of Theorem~\ref{Qld} is natural in $(A,\a,\p)$. To see this, suppose $\s:(A,\a,\phi)\to (B,\b,\psi)$ is a morphism. Then $\CP^r\circ \fix^r(\sigma)$ is the homomorphism $\sigma|\times G$, and it follows easily from \eqref{charquiggiso} that $\sigma\circ\t^A$ and $\t^B\circ(\sigma|\times G)$ agree on elements of the form $j_{\fix A}(c)j_G(f)$, and hence are equal.

Next we must show that
for every pair of objects $(A,\a,\phi)$ and $(B,\b,\psi)$,
\[
\fix^r:\mor((A,\a,\phi),(B,\b,\psi))\to\mor(\fix^r(A,\a,\p),\fix^r(B,\b,\psi))
\]
is a bijection.
To establish injectivity, suppose $\s,\tau:(A,\a,\phi)\to(B,\b,\psi) $ have $\s|=\tau|$. The last assertion in Theorem~\ref{Qld} implies that the elements $E(a)\phi(f)$ span a dense subspace of $A$, and
\[
\s(E(a)\phi(f))=\s|(E(a))\psi(f)=\tau|(E(a))\psi(f)=\tau(E(a)\phi(f)),
\]
so $\s=\tau$. For surjectivity, suppose that $\pi$ is a morphism from $\fix^r(A,\a,\phi)=(\fix A,\delta)$ to $\fix^r(B,\b,\psi)=(\fix B,\epsilon)$. Then the morphism $\pi\times G:(\fix A)\times_\delta G\to (\fix B)\times_\epsilon G$ satisfies
\[
\pi\times G(j_{\fix A}(c)j_G(f))=j_{\fix B}(\pi(c))j_G(f),
\]
and pulling this over under the isomorphisms $\theta$ of Theorem~\ref{Qld} gives a morphism $\s:(A,\a,\phi)\to (B,\b,\psi)$ such that
\begin{equation}\label{charsigma}
\s(c\phi(f))=\pi(c)\psi(f) \righttext{for $c\in \fix A$ and $f\in C_0(G)$.}
\end{equation}
But $\s(c\phi(f))=\s|(c)\psi(f)$, so \eqref{charsigma} implies that $\s|(c)=\pi(c)$, and we have $\s|=\pi$.
\end{proof}

Coaction cognoscenti usually work with full coactions rather than reduced ones, and might prefer to know the following analogue of \thmref{Thmfromcldc} for the classes of normal and maximal coactions. Our notation is explained in the Appendix.

\begin{cor}\label{normmax}
The crossed-product functors
\begin{align*}
&\CP^n:\CC^n(G)\to (C_0(G),\rt)\dn \AA(G)\righttext{and}\\
&\CP^m:\CC^m(G)\to (C_0(G),\rt)\dn \AA(G)
\end{align*}
are equivalences.
\end{cor}

\begin{proof}
Since $\CP^r$ is a quasi-inverse for $\fix^r$, it is an equivalence. So it follows from Theorem~\ref{boilerplate} that $\CP^n=\CP^r\circ\red$ is an equivalence. Now Corollary~\ref{CPOKmax} implies that $\CP^m=\CP^n\circ\nor$ is also an equivalence.
\end{proof}

\section{Iterated Landstad duality}
\label{compare}

Landstad duality for actions, as formulated in \cite[Theorem~4.1]{clda}, gives an equivalence between $\AA(G)$ and a comma category of coactions. When we apply the equivalence of \thmref{Thmfromcldc} to this comma category, we obtain an equivalence between $\AA(G)$ and an iterated comma category in $\AA(G)$. In this section, we identify this iterated comma category and obtain interesting new information about Landstad duality for actions (see Remark~\ref{newinfo}).

We begin with two abstract lemmas about comma categories which will help us identify the iterated comma category. The first is similar to \cite[Corollary~2.2]{clda}.

\begin{lem}
\label{comma equivalence}
Let $F:C\to D$ be a category equivalence, and let $a\in \obj C$.
Then
the map $\wilde F:a\dn C\to Fa\dn D$, defined on objects by
$\wilde F(x,f)=(Fx,Ff)$ and on morphisms by $\wilde Fh=Fh$,
is an equivalence.
\end{lem}

\begin{proof}
It is routine to check that $\wilde F$ is a functor. As usual, we verify that $\wilde F$ is full, faithful, and essentially surjective.
For the essential surjectivity, let $(y,g)\in \obj Fa\dn D$.
Since $F$ is an
equivalence, there exists $x\in \obj C$ and an isomorphism $\t:y\to
Fx$ in $D$, and then for the same reason there exists $f\in
C(a,x)$ such that $\t\circ g=Ff$. Then
\[\t:(y,g)\iso (Fx,Ff)=\wilde
F(x,f)\midtext{in}Fa\downarrow D.\]

To see that $\wilde F$ is full and faithful, let $k:\wilde F(x,f)\to \wilde F(y,g)$ in $Fa\dn D$.
We must show that there is a unique $h:(x,f)\to (y,g)$ in $a\dn C$ such that $\wilde F h=k$.
We have $k:Fx\to Fy$ in $D$, so because $F$ is an equivalence there is a unique $h:x\to y$ in $C$ such that $Fh=k$.
Since $k\circ Ff=Fg$ in $D$, we have $h\circ f=g$ in $C$, again because $F$ is an equivalence. Thus $h:(x,f)\to (y,g)$ in $a\dn C$ and $\wilde Fh=k$. Moreover, $h$ is the unique such morphism in $a\dn C$, because of its uniqueness in $C$.
\end{proof}

\begin{lem}\label{comma-comma}
Let $D$ be a category, let $b\in\obj D$,
and let $(y,g)\in \obj (b\dn D)$.
Then the iterated comma category
$(y,g)\dn(b\dn D)$
is isomorphic to $y\dn D$.
\end{lem}

\begin{proof}
Let $E=(y,g)\dn(b\dn D)$.
An object $((z,h),u)$ of $E$ comprises an object $z$ of $D$, a morphism $h\colon b\to z$ in $D$,
and a morphism $u\colon y\to z$ in $D$ such that $u\circ g = h$. A morphism $\phi\colon((z,h),u)\to((w,k),v)$ in $E$
is a morphism $\phi\colon z\to w$ in $D$
such that $\phi\circ h = k$ and $\phi\circ u = v$, and since $h=u\circ g$ and $k=v\circ g$, the equation $\phi\circ h = k$ is automatic given $\phi\circ u = v$. Thus the formulas
\[
G((z,h),u) = (z,u)
\quad\text{and}\quad
G(\phi) = \phi
\]
give well-defined maps on objects and morphisms from $E$ to $y\dn D$, and it is routine to check that $G$ is then a functor. The map
defined on objects and morphisms of $y\dn D$ by
\[
(z,u)\mapsto ((z,u\circ g),u)
\quad\text{and}\quad
\phi\mapsto\phi
\]
is an inverse for $G$, and thus $G$ is the desired isomorphism.
\end{proof}

For a locally compact group $G$, the canonical map $i_G^r:G\to UM(A\times_{\a,r}G)$ has an integrated form $i^{f,r}_G:C^*(G)\to M(A\times_{\a,r}G)$, which factors through a homomorphism $i^r_G:C_r^*(G)\to M(A\times_{\a,r}G)$ (so that $i^{f,r}_G=i^r_G\circ \lambda$).

\begin{prop}\label{comma-cor}
For any locally compact group $G$,
the assignments
\[
(A,\a)\mapsto
\bigl(A\times_{\a,r} G\times_{\what\a^r} G,\what{\what\a^r},i_G^r\times G\bigr)
\and
\p\mapsto \p\times_r G\times G
\]
give a category equivalence
\[
\AA(G)\sim
\bigl(C^*_r(G)\times_{\d_G^r} G,\what{\d_G^r}\;\bigr)\dn \AA(G).
\]
\end{prop}

\begin{proof}
Theorem~4.1 of \cite{clda} says that the assignments
\[
(A,\a)\mapsto (A\times_{\a,r} G,\what\a,i_G^{f,r})\and\phi\mapsto \phi\times_r G
\]
give an equivalence between $\AA(G)$ and $(C^*(G),\d_G)\dn\CC^n(G)$. There is a subtlety here: $(C^*(G),\d_G)$ need not be a normal coaction, and the arrows in the comma category are morphisms in $\CC(G)$. However, the normalization of $\d_G$ is the full coaction $\d_G^n$ of $G$ on $C_r^*(G)$ \cite[Proposition~A.61]{enchilada}, and every morphism $\phi:(C^*(G),\d_G)\to (B,\d)$ with $\d$ normal has the form $\phi^n\circ\lambda$ for a unique morphism $\phi^n:(C_r^*(G),\d_G^n)\to (B,\d)$. (The existence of $\phi^n$
is proved in the Appendix as part of the assertion that normalization is a functor.) Thus the map $(B,\d,\psi)\mapsto (B,\d,\psi\circ\lambda)$ is an isomorphism
\[
(C^*_r(G),\d_G^n)\dn\CC^n(G)\to (C^*(G),\d_G)\dn\CC^n(G)
\]
of comma categories. Thus
\[
(A,\a)\mapsto (A\times_{\a,r} G,\what\a,i_G^{r})\and\phi\mapsto \phi\times_r G
\]
is a category equivalence between $\AA(G)$ and $(C^*_r(G),\d_G^n)\dn\CC^n(G)$, which we denote (locally) by $\RCP$.

The isomorphism $\red$ of Theorem \ref{boilerplate} carries $(C^*_r(G),\d_G^n)$ to $(C^*_r(G),\d_G^r)$, and hence by \lemref{comma equivalence} induces an equivalence
\[
{\red}\;{\wilde{}}:(C^*_r(G),\d_G^n)\dn\CC^n(G)\to (C^*_r(G),\d_G^r)\dn\CC^r(G).
\]
The quasi-inverse equivalence $\CP^r:\CC^r(G)\to (C_0(G),\rt)\dn\AA(G)$ of Theorem~\ref{Thmfromcldc} carries $(C_r^*(G),\d_G^r)$ to $(C_r^*(G)\times_{\d_G^r}G,\what{\d_G^r},j_G)$, and hence by Lemma~\ref{comma equivalence} induces an equivalence of $(C_r^*(G),\d_G^r)\dn\CC^r(G)$ with the iterated comma category
\[
(C^*_r(G)\times_{\d_G^r}G,\what{\d_G^r},j_G)
\dn\big((C_0(G),\rt)\dn\AA(G)\big).
\]
Thus by Lemma~\ref{comma-comma}, $\CP^r$ induces an equivalence
\[
\dn\!\!\CP^r:(C^*_r(G),\d_G^r)\dn\CC^r(G)\to (C_r^*(G)\times_{\d_G^r}G,\what{\d_G^r})\dn\AA(G).
\]
The composition $\dn\!\!\CP^r\circ{\red}\;{\wilde{}}\circ\RCP$ is the desired equivalence.
\end{proof}

\begin{rem}\label{newinfo}
This result contains new information about Landstad duality. It implies that a system $(C,\b)$ in $\AA(G)$ is isomorphic
to a double-dual action if and only if there exists a morphism $\phi:(C^*_r(G)\times_{\d_G^r}G,{\what{\d_G^r}})\to (C,\b)$. The underlying homomorphism comes from a covariant homomorphism $(\pi,\mu):=(\phi\circ j_{C^*_r(G)},\phi\circ j_G)$ of $(C^*_r(G),\d_G^r)$ in $M(C)$; $\pi$ is the integrated form of a unitary $U:G\to UM(C)$, and the covariance of $(\pi,\mu)$ is then equivalent to the covariance of $(\mu,U):(C_0(G),\lt)\to M(C)$
(see, for example, Example~2.9(1) of \cite{rae:representation}).
So Proposition~\ref{comma-cor} implies that
$(C,\b)$ is isomorphic
to a double-dual action if and only if there is a covariant homomorphism $(\mu,U)$ of $(C_0(G),\lt)$ in $M(C)$ such that $\mu\circ\rt_s=\b_s\circ\mu$ and $\b_s(U_t)=U_t$ for all $s,t\in G$. It seems to us that it might take some work to deduce this assertion directly from the non-categorical versions of Landstad duality in \cite{lan} and~\cite{Q:landstad}.
\end{rem}

\begin{rem}\label{normmax2}
There are several other versions of iterated Landstad duality.
First,
an argument parallel to that of \propref{comma-cor}
gives a dual
equivalence
\[
(B,\d)\mapsto \bigl(B\times_{\d} G\times_{\what\d,r}G, \what{\what\d}, j_G^B\times G\bigr)\and \phi\mapsto \phi\times G\times_r G
\]
between $\CC^r(G)$ and $(C_0(G)\times_{\rt}G,\what{\rt})\dn \CC^r(G)$
(here we used the equality $C_0(G)\times_{\rt}G=C_0(G)\times_{\rt,r}G$).
Other variants can then be
derived using the equivalence $\CC^m(G)\sim \CC^n(G)$ and the isomorphism $\CC^n(G)\cong \CC^r(G)$.
\end{rem}

\section{Naturality of Mansfield imprimitivity}\label{sec-man}

Suppose that $H$ is a closed subgroup of a locally compact group $G$. In this section, we apply Theorem~\ref{XAnatequiv},
with $(X,G)=(G,H)$,
to elements of the comma category $(C_0(G),\rt|)\dn \AA(H)$ of the form $(B\times_\d G,\hat\d|,j_G)$, where $\d$ is a reduced coaction.

Theorem~3.1 of \cite{HR:mansfield} says that the restriction $\hat\d|$ of the dual action to $H$ is proper with respect to Mansfield's subalgebra $\DD$ of $B\times_\d G$. Lemma~3.3 of \cite{HR:mansfield} says that $\DD\subset \mm$, and the right inner product used in \cite{HR:mansfield} is defined using our averaging process $E$ as $\langle x,y\rangle=E(x^*y)$, so that the generalized fixed-point algebra $(B\times_\d G)^{\hat\d|}$ obtained in \cite{HR:mansfield} is the closed span in $M(B\times_\d G)$ of the set $\{E(x^*y):x,y\in\DD\}$. The proof of Theorem~3.1 in \cite[pages~1157--58]{HR:mansfield} shows that this generalized fixed-point algebra is precisely the subalgebra $B\times_{\d,r}(G/H)$ of $M(B\times_\d G)$.

We know from \cite[Lemma~3.2]{HR:mansfield} that $\DD=C_c(G)\DD C_c(G)$, and hence $\DD$ is a subalgebra of our $A_0:=C_c(G)(B\times_\d G)C_c(G)$. Thus each $E(x^*y)$
for $x,y\in \DD$
belongs to $E(A_0)$, and $(B\times_\d G)^{\hat\d|}$ is contained in our fixed-point algebra $\fix(B\times_\d G,\hat\d,j_G)$. On the other hand, $\DD$ is dense in $B\times_\d G$ by \cite[Theorem~12]{man}, and it follows from the norm continuity of the map $a\mapsto E(fag)$ that every $E(fag)$ for $fag\in A_0$ belongs to $(B\times_\d G)^{\hat\d|}$. This, together with
the discussion in the preceding paragraph, yields
\begin{equation}\label{fix=cp}
\fix(B\times_\d G,\hat\d,j_G)=(B\times_\d G)^{\hat\d|}=B\times_{\d,r}(G/H).
\end{equation}

\begin{prop}
The assignments
\begin{equation}\label{assign}
(B,\d)\mapsto B\times_{\d,r}(G/H)\and
\p\mapsto \p\times G|
\end{equation}
define a functor from $\CC^r(G)$ to $\CC$.
\end{prop}

\begin{proof}
We know from \thmref{Thmfromcldc}
that the assignments
\[
(B,\d)\mapsto (B\times_{\d}G,\what\d,j_G)\and
\p\mapsto \p\times G
\]
define a functor $\CP^r$ from $\CC^r(G)$ to the comma category $(C_0(G),\rt)\dn\AA(G)$ (indeed, this functor is a category equivalence). The map $(A,\a)\mapsto (A,\a|)$ is
easily seen to be
a functor from $\AA(G)$ to $\AA(H)$,
so by Lemma~\ref{comma equivalence} induces a functor on comma categories. We showed in
Corollary~\ref{fix functor}
that $\fix$ is a functor from $(C_0(G),\rt|)\dn\AA(H)$ to $\CC$. The assignments in \eqref{assign} are the composition of these three functors
\begin{equation}\label{decompofcpGH}
\CC^r(G)\xrightarrow{\CP^r} (C_0(G),\rt)\dn\AA(G)\to (C_0(G),\rt|)\dn\AA(H)\xrightarrow{\fix} \CC,
\end{equation}
and the result follows.
\end{proof}

Since the fixed-point algebra and the inner product in Proposition~\ref{is-proper} are the same as those used in \cite[Theorem~3.1]{HR:mansfield}, the bimodule $\overline{\DD}$ from \cite{HR:mansfield} embeds as a closed Hilbert submodule of $Z(B\times_\d G)$. Since $H$ acts freely on $G$, the imprimitivity algebras of both $\overline{\DD}$ and $Z(B\times_\d G)$ are $(B\times_\d G)\times_{\what\d,r} H$. Thus
the Rieffel correspondence implies that $\overline{\DD}=Z(B\times_\d G)$, and the following theorem on the naturality of the Morita equivalence in \cite{HR:mansfield} follows from
\remref{intoBEcat}.

\begin{thm}
\label{mansfield natural}
The bimodules $Z(B\times_\d G)$ implement a natural isomorphism between the functors $(B,\d)\mapsto B\times_{\d,r}(G/H)$ and $(B,\d)\mapsto (B\times_{\d}G)\times_{\hat\d,r}H$ from $\CC^r(G)$ to $\CC_{\aug}$.
\end{thm}

\begin{proof}
We just need to observe that the second functor factors as
\[
\CC^r(G)\xrightarrow{\CP^r} (C_0(G),\rt)\dn\AA(G)\to (C_0(G),\rt|)\dn\AA(H)
\xrightarrow{\RCP} \CC_{\aug},
\]
where the first two functors are the same as the first two in the factorization \eqref{decompofcpGH} of $(B,\d)\mapsto B\times_{\d,r}(G/H)$, and
the third is the functor $\RCP$ in \remref{intoBEcat}, which by Theorem~\ref{XAnatequiv} is naturally isomorphic to $\fix$.
\end{proof}

\begin{appendix}

\section{Reduction and normalization of coactions}

We consider
the category $\CC^r(G)$ in which the objects $(B,\d)$ are (nondegenerate) reduced  coactions $\d:B\to M(B\otimes C_r^*(G))$,
and
the category $\CC(G)$ of (nondegenerate) full coactions $\d:B\to M(B\otimes C^*(G))$; in both categories,  the morphisms   $\p:(B,\d)\to (C,\epsilon)$ are the nondegenerate homomorphisms $\p:B\to M(C)$ such that \begin{equation}\label{defequivar}(\p\otimes\id)\circ\d=\epsilon\circ \p,
\end{equation}
though in $\CC^r(G)$, \eqref{defequivar} holds in $M(C\otimes C_r^*(G))$, and in $\CC(G)$, it holds in $M(C\otimes C^*(G))$. For systems in $\CC(G)$, covariance is defined using the function $w_G:s\mapsto i_G(s)$ from $G$ to $UM(C^*(G))$, whereas for systems in $\CC^r(G)$, it is defined using $W_G:s\mapsto \lambda_s$ (see~\eqref{defcov}). Each system $(B,\d)$ has a crossed product $(B\times_{\d}G,j_B,j_G)$ carrying a dual action $\what\d$ of $G$, and there are functors $\CP$ on $\CC(G)$ and $\CP^r$ on $\CC^r(G)$ which take a system $(B,\d)$ to the element $(B\times_\d G,\what\d,j_G)$ of the comma category $(C_0(G),\rt)\dn \AA(G)$.

We will be interested in the full subcategory $\CC^n(G)$ of $\CC(G)$ consisting of the \emph{normal} coactions, that is, those $(B,\d)$ in $\CC(G)$ for which $j_B:B\to M(B\times_{\d}G)$ is an injection.  For every object $(B,\d)$ in $\CC^n(G)$, the map $\d^r:=(\id\otimes\lambda)\circ\d$ is a reduced coaction of $G$ on the same $C^*$-algebra $B$, called the \emph{reduction} of $\delta$ (see \cite[Proposition~3.3]{qui:fullred}). If $\p:(B,\d)\to (C,\epsilon)$ is a morphism in $\CC^n(G)$, then applying $\id\otimes\lambda$ to both sides of \eqref{defequivar} shows that the underlying homomorphism $\p:B\to M(C)$ also gives a morphism $\p:(B,\d^r)\to (C,\epsilon^r)$ in $\CC^r(G)$. Indeed, the assignments
\[
(B,\d)\mapsto (B,\d^r)\and
\p\mapsto \p
\]
form a functor $\red:\CC^n(G)\to \CC^r(G)$.

The following theorem sums up the properties of the reducing process.

\begin{thm}\label{boilerplate}
For every locally compact group $G$, the functor $\red$ is an isomorphism between $\CC^n(G)$ and $\CC^r(G)$. The crossed product functors $\CP^n:=\CP|_{\CC^n(G)}$ and $\CP^r$ are related by $\CP^n=\CP^r\circ\red$.
\end{thm}

There are choices implicit in the last assertion of the above theorem: a crossed product of a system $(B,\d)$ is by definition a triple $(C,j_B,j_G)$ which is universal for covariant homomorphisms, and to define the functors $\CP^n$ and $\CP^r$ we either need to nominate a particular construction which works for all $(B,\d)$ (as in \cite[\S3.1.2]{enchilada}), or choose a triple for each $(B,\d)$ using the axiom of choice (as advocated in \cite{clda}). The equality $\CP^n=\CP^r\circ\red$ means that, when we make a construction or choice for one of $\CP^n$ or $\CP^r$, that same construction or choice will work for the other.

\begin{proof}
To see that $\red$ is an isomorphism, it suffices to check that $\red$ is bijective on objects, faithful, and full.
If $(B,\d)$ is a reduced coaction, then Theorem~4.7 of \cite{qui:fullred} implies that there is a unique normal coaction $\d^f$ of $G$ on $B$ such that $\d=(\d^f)^r$, so $\red$ is bijective on objects. Since the underlying nondegenerate homomorphisms of $\p$ and $\red(\p)$ are the same, $\red$ is trivially faithful.
Checking fullness, though, seems to require some work:
we must show that if
$(B,\d)$, $(C,\epsilon)$ are objects in $\CC^n(G)$ and $\phi:(B,\d^r)\to (C,\epsilon^r)$ is a morphism in $\CC^r(G)$,
then $(\phi\otimes\id)\circ\d=\epsilon\circ\p$ as homomorphisms into $M(B\otimes C^*(G))$,
so that the underlying nondegenerate homomorphism $\phi:B\to M(C)$ is also a morphism in $\CC^n(G)$.

Theorem~4.7 of \cite{qui:fullred} says there is exactly one normal coaction with reduction $\d^r$, so we can identify $\d$ with the coaction $(\d^r)^f$ constructed in the proof of \cite[Theorem~4.7]{qui:fullred}. There it is proved that $\ad j_G\otimes\id(w_G)$ gives a full coaction of $G$ on the crossed product $B\times_{\d^r}G$, and that this coaction restricts to a full coaction on the subalgebra $j_B(B)$ of $M(B\times_{\d^r}G)$; $(\d^r)^f$ is the coaction on $B$ pulled back from the coaction $\ad j_G\otimes\id(w_G)$ via the isomorphism $j_B$. In other words, $\d=(\d^r)^f$ is characterized by
\begin{equation}\label{chardf}
j_B\otimes\id(\d(b))=\ad j_G\otimes\id(w_G)(j_B(b)\otimes 1).
\end{equation}
To help keep things straight, we denote the canonical map of $C_0(G)$ into $M(B\times_{\d^r}G)$ by $j_G^B$.

Since $\phi:(B,\d^r)\to (C,\epsilon^r)$ is a morphism in $\CC^r(G)$, there is a nondegenerate homomorphism $\phi\times G:B\times_{\d^r}G\to M(C\times_{\epsilon^r}G)$ such that
\[
\phi\times G(j_B(b)j_G^B(f))=j_C(\phi(b))j_G^C(f).
\]
This implies in particular that $(\phi \times G)\circ j_G^B=j_G^C$, so
\begin{equation}\label{phionwG}
(\p\times G)\otimes\id(j_G^B\otimes\id(w_G))=j_G^C\otimes\id(w_G).
\end{equation}
We now take $b\in B$ and compute, using $(\phi \times G)\circ j_B=j_C\circ\p$, \eqref{chardf}, \eqref{phionwG} and then \eqref{chardf} again:
\begin{align*}
j_C\otimes\id(\p\otimes\id(\d(b)))
&=(\p\times G)\otimes\id\big(j_B\otimes\id(\d(b))\big)\\
&=(\p\times G)\otimes\id\big(\ad j_G^B\otimes\id(w_G)(j_B(b)\otimes 1)\big)\\
&=\ad j_G^C\otimes\id(w_G)(j_C(\phi(b))\otimes 1)\\
&=j_C\otimes\id(\epsilon(\phi(b))),
\end{align*}
which, since $j_C$ is injective, implies that $\phi$ is a morphism in $\CC^n(G)$. Thus $\red$ is an isomorphism.

To finish, we deduce from
\cite[Proposition~3.8]{qui:fullred}
(or \cite[Theorem~4.1]{rae:representation}) that a triple $(C,\pi,\mu)$ is a crossed product of the reduction $(B,\d^r)$ if and only if it is a crossed product of $(B,\d)$, and hence $\CP^n=\CP^r\circ\red$.
\end{proof}

Next we consider the full subcategory $\CC^m(G)$ of $\CC(G)$ consisting of the \emph{maximal} coactions for which full crossed-product duality holds (see~\cite{ekrmax}). It was shown in \cite{clda} that normalization implements an equivalence between $\CC^m(G)$ and $\CC^n(G)$, and we want to know that this equivalence is compatible with crossed products.

If $(B,\d)$ is a full coaction, the \emph{normalization} introduced in \cite{qui:fullred} is a normal coaction $\d^n$ on a quotient $B^n$ of $B$; if $q_B:B\to B^n$ is the quotient map, then $\d^n$ is characterized by $\d^n\circ q_B=(q_B\otimes\id)\circ\d$. If $\phi:(B,\d)\to (C,\epsilon)$ is a morphism in $\CC(G)$, then there is a unique morphism $\phi^n:(B^n,\d^n)\to(C^n,\epsilon^n)$ in $\CC(G)$ such that $\phi^n\circ q_B=q_C\circ \phi$. (Since $(C^n,\epsilon^n)$ is normal, it has a covariant representation $(\pi,\mu)$ with $\pi$ faithful; now $(\pi\circ q_C\circ \phi,\mu)$ is a covariant representation of $(B,\d)$, and \cite[Proposition~2.6]{qui:fullred} implies that $\pi\circ q_C\circ\phi$ and $q_C\circ\phi$ factor through $q_B$.) Thus normalization is a functor $\nor:\CC(G)\to \CC^n(G)$.

\begin{prop}\label{CPOK}
The composition $\CP^n\circ\nor:\CC(G)\to (C_0(G),\rt)\dn\AA(G)$ is naturally isomorphic to $\CP$.
\end{prop}

\begin{proof}
If $(B,\d)$ is a full coaction, then the induced map $q_B\times G$ is an isomorphism of $B\times_\d G$ onto $B^n\times_{\d^n} G$ (see \cite[Proposition~2.6]{qui:fullred}); we claim that these isomorphisms are natural. Suppose $\phi:(B,\d)\to (C,\epsilon)$ is a morphism in $\CC(G)$. Then $\phi^n:B^n\to M(C^n)$ is also a nondegenerate homomorphism, and satisfies $\phi^n\circ q_B=q_C\circ \phi$. But now the functoriality of $\CP:\CC(G)\to (C_0(G),\rt)\dn\AA(G)$ implies that
\[
(\phi^n\times G)\circ (q_B\times G)=(q_C\times G)\circ (\phi\times G),
\]
and hence we have the naturality of the isomorphisms $q_B\times G$.
\end{proof}

\begin{cor}\label{CPOKmax}
Normalization implements an equivalence between $\CC^m(G)$ and $\CC^n(G)$, and $\CP^n\circ\nor$ is naturally isomorphic to $\CP^m:=\CP|_{\CC^m(G)}$.
\end{cor}

\begin{proof}
It is proved in \cite[Theorem~3.3]{clda} that $\nor$ is an equivalence on the subcategory $\CC^m(G)$, and the rest follows from Proposition~\ref{CPOK}.
\end{proof}

\end{appendix}


\providecommand{\bysame}{\leavevmode\hbox to3em{\hrulefill}\thinspace}
\providecommand{\MR}{\relax\ifhmode\unskip\space\fi MR }
\providecommand{\MRhref}[2]{%
  \href{http://www.ams.org/mathscinet-getitem?mr=#1}{#2}
}
\providecommand{\href}[2]{#2}

\end{document}